\documentclass{birkmult}
\usepackage{amsmath}
\usepackage{amsfonts}

\usepackage{times} 
\usepackage{mhequ}
\usepackage{Symbols}

\usepackage{ScriptFonts}

\newtheorem{lemma}{Lemma}[section]
\newtheorem{theorem}[lemma]{Theorem}

\theoremstyle{remark}

\newtheorem{remark}[lemma]{Remark}

\newcommand{\Norm}[1]{|\!|\!|#1|\!|\!|} 
\newcommand{\norm}[1]{\|#1\|}

\newcommand{\cQ}{\mathcal{Q}}
\newcommand{\cP}{\mathcal{P}}

\newcommand{\cS}{\mathcal{S}}

\newcommand{\dd}{\rho}
\renewcommand{\d}{d}

\newcommand{\X}{\mathbf{X}}
\newcommand{\ccdot}{\;\cdot\;} 
\renewcommand{\rho}{\orho}

 \newtheorem{assumption}{Assumption}

\renewcommand{\theenumi}{\roman{enumi}}
\renewcommand{\theenumii}{\roman{enumii}}
\renewcommand{\theenumiii}{\arabic{enumiii}}
\renewcommand{\labelenumi}{(\theenumi)}

\title[Yet another look at Harris' ergodic theorem for Markov chains]{Yet another look at Harris' ergodic theorem\\ for Markov chains}
\author{Martin~Hairer}
\address{Mathematics Institute, The University of Warwick, CV4
  7AL, UK}
\email{M.Hairer@Warwick.ac.uk}

\author[Jonathan Mattingly]{Jonathan~C.~Mattingly}
\address{Mathematics Department, Center of Theoretical and
  Mathematical Science, Department of Statistical Science, and Center
  of Nonlinear and Complex Systems, Duke University, Durham NC, USA}
\email{jonm@math.duke.edu}

\begin{document}
\maketitle

The aim of this note is to present an elementary proof of a variation
of Harris' ergodic theorem of Markov chains. This theorem, dating back
to the fifties \cite{HarrisESMMP56} essentially states that a Markov
chain is uniquely ergodic if it admits a ``small'' set (in a technical
sense to be made precise below) which is visited infinitely
often. This gives an extension of the ideas of Doeblin to the
unbounded state space setting. Often this is established by finding a
Lyapunov function with ``small'' level sets
\cite{Hasminskii80,MeynTweedieMC93}. If the Lyapunov function is
strong enough, one has a spectral gap in a weighted supremum norm
\cite{MeyTwe92:542,MeynTweedieMC93}. In particular, its transition
probabilities converge exponentially fast towards the unique invariant
measure, and the constant in front of the exponential rate is
controlled by the Lyapunov function \cite{MeynTweedieMC93}.

Traditional proofs of this result rely on the decomposition of the
Markov chain into excursions away from the small set and a careful
analysis of the exponential tail of the length of these excursions
\cite{Num84,Cha89:702,MeyTwe92:542,MeynTweedieMC93}. There have been other variations
which have made use of Poisson equations or worked at getting explicit
constants \cite{kontoyiannis_meyn_2005,douc_moulines_rosenthal_2004,del_moral_ledoux_miclo_2003}. The
present proof is very direct, and relies instead on introducing a
family of equivalent weighted norms indexed by a parameter $\beta$ and
to make an appropriate choice of this parameter that allows to combine
in a very elementary way the two ingredients (existence of a Lyapunov
function and irreducibility) that are crucial in obtaining a spectral
gap. Use of a weighted total-variation norm has been important since
\cite{MeyTwe92:542}.

The original motivation of this proof was the authors' work on spectral
gaps in Wasserstein metrics. The proof presented in this note is a
version of our reasoning in the total variation setting which we used
to guide the calculations in \cite{HaiMat08:??}.  While we initially
produced it for this purpose, we hope that it will be of interest in
its own right. 

\section{Setting and main result}

Throughout this note, we fix a measurable space $\X$ and a Markov transition kernel $\cP(x,\cdot)$ on $\X$.
We will use the notation $\cP$ for the operators defined as usual
on both the set of bounded measurable functions and the set of 
measures of finite mass by
\begin{equ}
\bigl(\cP \phi\bigr)(x) = \int_\X \phi(y)\,\cP(x,dy)\;,\quad \bigl(\cP \mu\bigr)(A) = \int_\X \cP(x,A)\,\mu(dx)\;.
\end{equ}
Hence we are using $\cP$ both to denote the action on functions and its duel action on measure. 
Note that $\cP$ extends trivially to measurable functions $\phi \colon \X \to [0,+\infty]$.
We first assume that $\cP$ satisfies the following geometric drift condition:

\begin{assumption}\label{ass:Harris2}
  There exists a function $V:\X \rightarrow [0,\infty)$ and constants $K\geq 0$ and
  $\gamma \in (0,1)$ such that
 \begin{equ}[e:Lyap]
   (\cP V)(x) \leq \gamma V(x) +K\;,
 \end{equ}
for all $x \in \X$.
\end{assumption}

\begin{remark}
  One could allow $V$ to also take the value $+\infty$. However, since
  we do not assume any particular structure on $\X$, this case can
  immediately be reduced to the present case by replacing $\X$ by
  $\{x\,:\, V(x) < \infty\}$.
\end{remark}
Assumption~\ref{ass:Harris2} ensures that the dynamics enters the
``center'' of the state space regularly with tight control on the
length of the excursions from the center.  We now assume that a
sufficiently large level set of $V$ is sufficiently ``nice'' in the
sense that we have a uniform ``minorization'' condition reminiscent of
Doeblin's condition, but localized to the interior of the level set.

\begin{assumption}\label{ass:Harris1}
  There exists a constant $\alpha
  \in(0,1)$   and a probability measure $\nu$ so that
  \begin{align*}
    \inf_{x \in \mathcal{C}} \cP(x,\ccdot) \geq \alpha \nu(\ccdot)\;,
  \end{align*}
  with $\mathcal{C} = \{ x \in \X : V(x) \leq R\}$ for some $R >
  2K/(1-\gamma)$ where $K$ and $\gamma$ are the constants from
  Assumption~\ref{ass:Harris2}.
\end{assumption}

In order to state the version Harris' theorem under consideration, we
introduce the following weighted supremum norm:
\begin{equ}[e:normphi]
  \norm{\phi}= \sup_{ x} \frac{|\phi(x)|}{1 + V(x)}\;.
\end{equ}
With this notation at hand, one has:

\begin{theorem}\label{theo:Harris}
If Assumptions~\ref{ass:Harris2} and
 \ref{ass:Harris1} hold,  then $\cP$ admits a unique invariant measure $\mu_\star$.
Furthermore,  there exist constants $C>0$ and $\gamma \in (0,1)$ such that the bound
\begin{equ}
\norm{\cP^n\phi - \mu_\star(\phi)} \le C \gamma^n \norm{\phi - \mu_\star(\phi)}
\end{equ}
holds for every measurable function $\phi\colon \X\to\R$ such that $\norm{\phi}<\infty$.
\end{theorem}

While this result is well-known, the proofs found in the literature are often quite involved
and rely on careful  estimates of the return times to small sets, combined with a clever application
of Kendall's lemma. See for example \cite[Section~15]{MeynTweedieMC93}.

The aim of this note is to provide a very short and elementary proof of
Theorem~\ref{theo:Harris} based on a simple trick.
Instead of working directly with \eref{e:normphi},  we define a whole family of weighted
supremum norms depending on a scale parameter $\beta >
0$ that are all equivalent to the original norm \eref{e:normphi}:
\begin{align*}
  \norm{\phi}_{\beta}= \sup_{ x} \frac{|\phi(x)|}{1 + \beta V(x)}\;.
\end{align*}
We also define the associated dual metric $\dd_\beta$ on probability measures given by
\begin{align}\label{eq:dd_def}
  \dd_{\beta}(\mu_1,\mu_2) &= \sup_{\phi : \norm{\phi}_\beta \leq 1} \int_\X
  \phi(x) \bigl(\mu_1- \mu_2\bigr)(dx)\;.
\end{align}
It is well-known that $\dd_\beta$ is nothing but a weighted total variation distance:
\begin{equ}
  \dd_{\beta}(\mu_1,\mu_2) = \int_\X \bigl(1+\beta V(x)\bigr)\,|\mu_1 - \mu_2|(dx)\;.
\end{equ}
With these notations, our main result is:

\begin{theorem}\label{thm:main} If Assumptions~\ref{ass:Harris2} and
 \ref{ass:Harris1} hold,  then there exists $\bar \alpha \in (0,1)$ and
  $\beta >0$ so that
\begin{equs}
    \dd_\beta( \cP \mu_1 , \cP \mu_2) &\leq \bar \alpha
    \dd_\beta(\mu_1,\mu_2)
\end{equs}
  for any probability measure $\mu_1$ and $\mu_2$ on $\X$. In
  particular, for any $\alpha_0 \in (0,\alpha)$ and $\gamma_0 \in
  (\gamma + 2K/R, 1)$ one can choose $\beta =\alpha_0/K$ and $\bar
  \alpha= (1- (\alpha-\alpha_0)) \vee (2+ R \beta\gamma_0)/(2+R
  \beta)$.
\end{theorem}

\begin{remark}
The interest of this result lies in the fact that it is possible to tune $\beta$ in such a way that 
$\cP$ is a strict contraction for the distance $\dd_\beta$. In general, this does \textit{not} imply
that  $\cP$ is a contraction for $\dd_1$, say, even though the equivalence of the norms
$\norm{\cdot}_\beta$ does of course imply that there exists
$n>0$ such that $\cP^n$ is such a contraction.
\end{remark}

\section{Alternative formulation of metric $\dd_\beta$}

We now introduce an alternative definition of the weighted total
variation norm $\dd_\beta$. We begin by defining a metric $\d_\beta$ between points in
$\X$ by
\begin{align*}
  \d_\beta(x,y)=
  \begin{cases}
    0 & x=y\\
    2 + \beta V(x) + \beta V(y) & x \neq y
  \end{cases}
\end{align*}
Though sightly odd looking, the reader can readily verify that since $V\ge 0$, 
$d_\beta$ indeed satisfies the axioms of a metric.  This metric in turn
induces a Lipschitz seminorm on measurable functions and a metric on probability
measures defined respectively by
\begin{align*}
  \Norm{\phi}_{\beta} &= \sup_{x \neq y} \frac{|\phi(x)
    -\phi(y)|}{\d_\beta(x,y)}\;,\\
  \d_\beta(\mu_1,\mu_2) &= \sup_{\phi : \Norm{\phi}_{\beta} \leq 1} \int
  \phi(x) \bigl(\mu_1 - \mu_2\bigr)(dx)\;.
\end{align*}
It turns out that these norms are almost identical to the ones from the previous section.
More precisely, one has:
\begin{lemma}\label{lem:reduction}
One has the identity $\Norm{\phi}_\beta = \inf_{c \in \R} \norm{\phi + c}_\beta$.
In particular, $\d_\beta = \dd_\beta$.
\end{lemma}
\begin{proof}
It is obvious that $\Norm{\phi}_\beta \le \norm{\phi}_\beta$ and therefore 
$\Norm{\phi}_\beta \le \inf_{c \in \R} \norm{\phi + c}_\beta$, so it remains to show
the reverse inequality.

  Given any $\phi$ with $\Norm{\phi}_{\beta} \leq 1$, we set $c =
  \inf_x \bigl(1 + \beta V(x) - \phi(x)\bigr)$. Observe that for any $x$ and $y$,
  $\phi(x) \leq |\phi(y)| + |\phi(x) - \phi(y)| \leq |\phi(y)| + 2
  +\beta V(x) +\beta V(y)$. Hence $1+ \beta V(x) - \phi(x) \geq -1
  -\beta V(y) - |\phi(y)|$. Since there exists at least one point with $V(y) <
  \infty$ we see that $c$ is bounded from below and hence $|c| <
  \infty$.

  Observe now that
  \begin{align*}
    \phi(x) +c \leq \phi(x) + 1+\beta V(x) - \phi(x) = 1+ \beta V(x)\;,
  \end{align*}
  and 
  \begin{align*}
    \phi(x) +c &=\inf_y \phi(x) + 1+\beta V(y) - \phi(y)\\
    &\geq \inf_y 1 + \beta V(y) - \Norm{\phi}_\beta\cdot \d_\beta(x,y)
    \geq - (1 + \beta V(x))\;,
  \end{align*}  
  so that $| \phi(x) + c| \leq 1+ \beta V(x)$ as required.
  
It follows that the sets $\{\phi: \norm{\phi}_\beta \le 1\}$ and
$\{\phi : \Norm{\phi}_\beta \le 1\}$ only 
differ by additive constants, so that one has indeed $\d_\beta = \dd_\beta$.
\end{proof}

\begin{remark}
Note that of course $\d_\beta = \dd_\beta$ only for probability measures, or at least positive 
measures of equal mass. Otherwise, $\d_\beta$ is $+\infty$ in general, while $\dd_\beta$ need not be.
\end{remark}

\section{Proof of main theorem}

\begin{theorem} If Assumptions~\ref{ass:Harris2} and
  \ref{ass:Harris1} hold there exists 
  an
 $\bar \alpha \in (0,1)$ and $\beta >0$ such that 
  \begin{align*}
\Norm{\cP\phi}_\beta \le \bar\alpha \Norm{\phi}_\beta\;.
  \end{align*}
Actually, setting $\gamma_0 = \gamma + 2K/R < 1$, for any $\alpha_0  \in
(0,\alpha)$  one can choose $\beta =\alpha_0/K$ and  $\bar \alpha=
(1- \alpha+\alpha_0) \vee (2+ R \beta\gamma_0)/(2+R \beta)$.
\end{theorem}
\begin{proof} 
Fix a test function $\phi$ with $\Norm{\phi}_\beta \le 1$. 
By Lemma~\ref{lem:reduction}, we can assume without loss of generality that one
also has $\norm{\phi}_\beta \le 1$. The claim then follows if we can exhibit $\bar\alpha < 1$
so that 
\begin{equ}
  |\cP\phi(x) - \cP\phi(y)| \le \bar \alpha \d_\beta(x,y)\;.
\end{equ}

  If $x=y$, the claim is true. Henceforth we assume $x \neq y$.
  We begin by assuming that $x$ and $y$ are
  such that
  \begin{align}\label{outSide}
    V(x) + V(y) \geq R\ .
  \end{align}
Fixing $\gamma_0$ as in the statement of the theorem, for any $\beta>0$
  we set
  $\gamma_1=(2+ \beta R \gamma_0)/(2+\beta R)$. Observe that for $\beta \in
  (0,1)$ and $R>0$, one has $\gamma_1 \in (\gamma_0,1)$. With these
  choices, we have from \eref{e:Lyap} and \eref{e:normphi} the bound
\begin{align*}
  |\cP\phi(x) - \cP\phi(y)| & \leq 2 +\beta\cP V(x) + \beta\cP V(y)\\
&\leq  2 + \beta \gamma V(x) + \beta
  \gamma V(y) + 2 \beta K\\
  &\leq 2 +  \beta \gamma_0 V(x) + \beta
  \gamma_0 V(y) \\
  &\leq 2\gamma_1 +  \beta \gamma_1 V(x) + \beta
  \gamma_1 V(y) = \gamma_1 \d_\beta(x,y)\,.
\end{align*}
The third line follows from our choice of $\gamma_0$ and the fact that
by \eqref{outSide} we know that $2 K \leq (\gamma_0 -\gamma)(V(x) +
V(y))$.  The last line follows from the fact that $2(1-\gamma_1) =
\beta R (\gamma_1 -\gamma_0)\leq\beta  (\gamma_1 -\gamma_0)(V(x) +
V(y))$ given our choice of $\gamma_1$. We emphasise that up to now
$\beta$ could be any positive number; only the precise value of 
$\gamma_1$ depends on it (and gets ``worse'' for small values of $\beta$). 
The second part of the proof will determine 
a choice of $\beta >0$.

Now consider the case of $x$ and $y$ such that $V(x) + V(y) \leq R$
and hence $x,y \in \mathcal{C}$. For such $x$ and $y$ we define the
Markov transition $\tilde \cP$ by $\tilde \cP(x,\ccdot)
=\frac{1}{1-\alpha} \cP(x,\ccdot) -\frac{\alpha}{1-\alpha}
\nu(\ccdot)$. Now we have $\cP \phi(x) = (1-\alpha) \tilde \cP \phi(x)
+ \alpha \int \phi d\nu$ and $\cP \phi(y) = (1-\alpha) \tilde \cP
\phi(y) + \alpha \int \phi d\nu$. Subtracting the second of these
expressions from the first and using that since $V$ is a non-negative
function $\tilde \cP V(x) \leq \frac{1}{1-\alpha} \cP V(x)$ produces
\begin{align*}
 |\cP \phi(x) - \cP \phi(y)| &=    (1-\alpha) \big| \tilde \cP
 \phi(x) - \tilde \cP \phi(y) \big| \\
&\leq (1-\alpha) 2 + (1-\alpha)\beta( \tilde \cP V(x)  +\tilde \cP
V(y))\\
&\leq (1-\alpha) 2 + \beta(\cP V(x)  +\cP
V(y))\\
& \leq (1-\alpha)2  + \gamma \beta V(x) + \gamma \beta V(y) + 2\beta
K\;.
\end{align*}
Hence fixing $\beta=\alpha_0/K$ for any $\alpha_0  \in
(0,\alpha)$ and setting and $\gamma_2=
(1-(\alpha-\alpha_0))\vee \gamma\in (0,1)$ produces 
\begin{align*}
 |\cP \phi(x) - \cP \phi(y)| 
&\leq 2(1-(\alpha-\alpha_0)) +  \gamma \beta V(x) +  \gamma \beta V(y)\\
&\leq \gamma_2 \d_\beta(x,y)\,.
\end{align*}
Setting $\bar \alpha = \gamma_1 \vee \gamma_2$ and recalling that
$\gamma_1 \geq \gamma$ concludes the proof.
\end{proof}

Theorem~\ref{thm:main} now follows as a corollary since $\d_\beta = \dd_\beta$
and $\d_\beta$ is the norm dual to $\Norm{\cdot}_\beta$. In order to conclude that Theorem~\ref{theo:Harris}
holds, it only remains to show that our assumptions imply that an invariant measure $\mu_\star$ actually exists and
that the integral of $V$ with respect to $\mu_\star$ is finite.

\subsection{Existence of an invariant measure}

We have already shown that  Assumptions~\ref{ass:Harris2} and
\ref{ass:Harris1} allow to prove that for some $\beta>0$, $\cP$ is a strict contraction in the 
weighted total variation metric $\dd_\beta$  defined by \eqref{eq:dd_def}.
We now show that the same assumptions are also sufficient to ensure the existence of an invariant measure:
\begin{theorem}
   If Assumptions~\ref{ass:Harris2} and
 \ref{ass:Harris1} hold then there exists a probability measure $\mu_\infty$ on $\X$
 such that $\int V\,d\mu_\infty < \infty$ and which is invariant in that $\cP \mu_\infty = \mu_\infty$.
\end{theorem}
\begin{proof} Fixing any $x \in \X$, for $n \in \N$\ define
  $\mu_n=\cP^n \delta_x$. By Theorem~\ref{thm:main}, we know that for
  some $\bar\alpha \in (0,1)$ and some $\beta>0$,
  \begin{align*}
    \dd_\beta(\mu_{n+1},\mu_{n})\leq \alpha^n
    \dd_\beta(\mu_{1},\delta_x)\ .
  \end{align*}
  Hence, $\mu_{n}$ is a Cauchy sequence. Since $\dd_\beta$ is complete for
  the space of probability measures integrating $V$ (because the total variation distance is complete for 
  the space of measures with finite mass) there
  exists a probability measure $\mu_\infty$ so that $\dd_\beta(\mu_{n},\mu_\infty)
  \rightarrow 0$ as $n \rightarrow \infty$. Since this implies that $\mu_n \to \mu_\infty$ in total variation
  and $\cP$ is always a contraction in the
  total variation distance, it follows that $\cP \mu_\infty =\lim \cP \mu_n = \lim \mu_{n+1} =
  \mu_\infty$ as required.
\end{proof}
\subsection{A slightly different set of assumptions}

Many results in the theory of Harris chains results are proved under
a slightly different set of assumptions. The Lyapunov function
condition in Assumption~\ref{ass:Harris2} is replaced with the following:
\begin{assumption}\label{ass:Harris2p}
  There exists a function $V:\X \rightarrow [1,\infty)$ and constants $b\geq 0$,
  $\tilde \gamma \in (0,1)$ and a subset $S \subset \X$ such that
 \begin{equ}[e:Lyap2]
   (\cP V)(x) \leq \tilde \gamma V(x) + b \one_S(x)\;,
 \end{equ}
 for all $x \in \X$.
\end{assumption}

Clearly Assumption~\ref{ass:Harris2p} implies
Assumption~\ref{ass:Harris2} with $K=b$. The question is whether
Assumption~\ref{ass:Harris1} holds with that choice of $K$ and with
$\mathcal{C}$  defined as in Assumption~\ref{ass:Harris1}.  If it does
then our main theorem holds. However, Assumption~\ref{ass:Harris2p}
is most naturally paired with the following modified version of
Assumption~\ref{ass:Harris1}.
\begin{assumption}\label{ass:Harris1p}
  There exists a constant $\tilde \alpha
  \in(0,1]$ and a probability measure $\tilde \nu$ so that the lower bound
  \begin{align*}
    \inf_{x \in S} \cP(x,\ccdot) \geq \tilde\alpha \tilde\nu(\ccdot)
  \end{align*}
holds. Here, the set $S$ is the same as in Assumption~\ref{ass:Harris2p}.
\end{assumption}

It is relatively clear that Assumptions~\ref{ass:Harris2} and
\ref{ass:Harris1} together imply Assumptions~\ref{ass:Harris2p} and
\ref{ass:Harris1p}. In particular, if one picks a $\bar \gamma \in
(\gamma,1)$ sufficiently close to one, then $R \geq K/(\bar \gamma-
\gamma)$ and setting $S= \{ x: V(x) \leq K\}$ we see that the desired
implication holds.

\begin{remark}
  In general, one cannot hope for Assumptions~\ref{ass:Harris1p} and
  \ref{ass:Harris2p} to imply Assumptions~\ref{ass:Harris2} and
  \ref{ass:Harris1} and hence the existence of a spectral gap without
  any further assumptions. A trivial example is given by $\X = \{0,1\}$
  with the (deterministic) transition probabilities $\cP(x,\cdot) =
  \delta_{1-x}$. This Markov operator has spectrum $\{-1,1\}$ and has
  therefore no spectral gap. On the other hand,
  Assumptions~\ref{ass:Harris1p} and \ref{ass:Harris2p} are satisfied
  with $\tilde \alpha = 1$, $\tilde \gamma = 1/2$, and $b = 3/2$ if one makes for
  example the choice $S = \{0\}$, $\tilde \nu = \delta_1$,  and $V(x) = 1+ x$.
\end{remark}

In spite of the preceding remark, we are now going to show that
Assumptions~\ref{ass:Harris1p} and \ref{ass:Harris2p} are essentially
equivalent to Assumptions~\ref{ass:Harris2} and \ref{ass:Harris1} from
the previous section. More precisely, for $N>0$,
define the ``averaged'' Markov operator
\begin{equ}
\cQ = {1 \over N+1}\sum_{k=0}^N \cP^k\;.
\end{equ}
Then we have:
\begin{theorem}
If $\cP$ satisfies Assumptions~\ref{ass:Harris1p} and \ref{ass:Harris2p}, then there exists a choice of $N$ such that
$\cQ$ satisfies Assumptions~\ref{ass:Harris2} and \ref{ass:Harris1}.
\end{theorem}

\begin{proof}
Fix some arbitrary $R$ with
$R > 2b/(1-\tilde\gamma)$. Our aim is to show that we can find $N>0$, a probability measure $\nu$ and a constant
$\alpha > 0$ such that $\cQ(x,\cdot) \ge \alpha \nu(\cdot)$ for every $x$ with $V(x) \le R$.

Iterating \eref{e:Lyap2}, we find that one has the bound
\begin{equ}[e:lowerboundPS]
1 \le \cP^{n+1} V \le \tilde \gamma^{n+1} V + b \sum_{k=0}^n \tilde \gamma^k \cP^{n-k} \one_S\;,
\end{equ}
so that, on the set $\cS_n = \{x\,:\, V(x) \le \tilde\gamma^{-n-1}/2\}$, one has the lower bound
\begin{equ}[e:lowerBound2]
\inf_{x \in \cS_n} \sum_{k=0}^n \tilde \gamma^k \cP^{n-k}(x, S) \ge {1\over 2b}\;.
\end{equ}
In particular this implies that for every $x \in \X$, there exists $n$ such that $\cP^n(x,S) > 0$. Combining this with
our two assumptions shows that $\int V(x)\tilde\nu(dx) = C < \infty$ so that, integrating \eref{e:lowerboundPS} with respect to $\tilde\nu$,
we obtain
\begin{equ}
1 \le C \tilde \gamma^{n+1} + b \sum_{k=0}^n \tilde \gamma^k \bigl(\cP^{n-k}\tilde\nu\bigr)(S)\;.
\end{equ}
Choosing $n$ sufficiently large then implies the existence of some $\ell > 0$ such that $\bigl(\cP^{\ell-1}\tilde\nu\bigr)(S) > 0$.
Combining this with Assumption~\ref{ass:Harris2p} shows that there exists $\hat \alpha > 0$ such that $\cP^{\ell}\tilde\nu \ge \hat \alpha \tilde \nu$.
Setting now $\nu = {1\over \ell} \sum_{k=0}^{\ell-1} \cP^{k}\tilde \nu$, it follows that one has the bound
\begin{equ}
\cP \nu = {1\over \ell} \sum_{k=1}^{\ell-1} \cP^{k}\tilde \nu + {1\over \ell} \cP^{\ell}\tilde \nu
\ge {1\over \ell} \sum_{k=1}^{\ell-1} \cP^{k}\tilde \nu + {\hat \alpha \over \ell} \tilde \nu \ge \hat \alpha \nu\;.
\end{equ}
In particular, this implies that for every $m \ge 1$ there exists a constant $\alpha_m$ such that the lower bound
\begin{equ}[e:lowerBound3]
\inf_{x \in S} \sum_{k=m}^{m+\ell} \cP^k(x, \cdot) \ge \alpha_m \nu(\cdot)
\end{equ}
holds.  Let now $n$ be sufficiently large such that $\tilde\gamma^{-n-1}/2 \ge R$ and set $N = n+1+\ell$.
Combining \eref{e:lowerBound2} and \eref{e:lowerBound3} then yields the desired result.
\end{proof}

\begin{remark}
Keeping track of the constants appearing in the proof of the previous result, we see that one can choose for example
any integer $N$ such that
\begin{equ}
N > 1 + \log \Bigl({2b \over 1-\tilde\gamma} \int V(x)\,\tilde \nu(dx)\Bigr)/\log \tilde\gamma\;.
\end{equ}
\end{remark}

\subsection*{Acknowledgment}
MH acknowledges the support of an EPSRC Advanced Research Fellowship
(grant number EP/D071593/1) and JCM the support of NSF through a
PECASE award (DMS-0449910) and the support of the Sloan foundation
through a Sloan foundation fellowship.

\bibliographystyle{Martin}
\bibliography{./refs}

\end{document}